\documentclass[11pt,a4paper]{article}
\usepackage{amsmath,amssymb,amsfonts,graphicx,amsthm}

\newtheorem{thm}{{\bf  Theorem}}
\newtheorem{cor}{{\bf  Corollary}}
\newtheorem{prepos}{{\bf  Preposition}}

\newtheorem{preexample}{{\bf Example}}

\newtheorem{preconj}{{\bf  Conjecture}}

\newtheorem{preremark}{{\bf  Remark}}

\newtheorem{lem}{{\bf  Lemma}}

\begin{document}

\title{\bf On edge-decomposition of cubic graphs into copies of the double-star with four edges 
\thanks {2010{ \it Mathematics Subject Classification}: 05C51, 05C05
 }}

\author{{\normalsize
{\sc S. Akbari${}^{\mathsf{a}, \mathsf{c}}$},\,
  {\sc H. R. Maimani${}^{\mathsf{b}, \mathsf{c}}$}\,
  {\sc and A. Seify${}^{\mathsf{b}, \mathsf{c}}$},\,}
 \vspace{3mm}
\\{\footnotesize{${}^{\mathsf{a}}$\it Department of
Mathematical Sciences, Sharif University of Technology, Tehran,
Iran}}
{\footnotesize{}}\\{\footnotesize{${}^{\mathsf{b}}$\it Department of Science, Shahid Rajaee Teacher Training University, Tehran, Iran}}
{\footnotesize{}}\\{\footnotesize{${}^{\mathsf{c}}$\it School of Mathematics, Institute for Research in Fundamental Sciences (IPM),}}{\footnotesize{}}\\{\footnotesize{${}^{\mathsf{}}$\it
P.O. Box 19395-5746,
 Tehran, Iran.}}
\thanks{{\it E-mail addresses}: $\mathsf{s\_akbari@sharif.edu}$, $\mathsf{maimani@ipm.ir}$ and $\mathsf{abbas.seify@gmail.com}$.} }

\date{}

\maketitle

\begin{abstract}
A tree containing exactly two non-pendant vertices is called a double-star. Let $k_1$ and $k_2$ be two positive integers. The double-star with degree sequence $(k_1+1, k_2+1, 1, \ldots, 1)$ is denoted by $S_{k_1, k_2}$. If $G$ is a cubic graph and has an $S$-decomposition, for a double-star $S$, then $S$ is isomorphic to $S_{1,1}$, $S_{1,2}$ or $S_{2,2}$. It is known that a cubic graph has an $S_{1,1}$-decomposition if and only if it contains a perfect matching. In this paper, we study the $S_{1,2}$-decomposition of cubic graphs. First, we present some necessary conditions for the existence of an $S_{1, 2}$-decomposition in cubic graphs. Then we prove that every $\{C_3, C_5, C_7\}$-free cubic graph of order $n$ with $\alpha(G)=\frac{3n}{8}$ has an $S_{1, 2}$-decomposition, where $\alpha(G)$ denotes the independence number of $G$. Finally, we obtain some results on the $S_{1, r-1}$-decomposition of $r$-regular graphs. 
\end{abstract}

\section{Introduction}
Let $G=(V(G),E(G))$ be a graph and $v \in V(G)$. We denote the set of all neighbors of $v$ by $N(v)$ and for $X \subseteq V(G)$ we define $N(X)= \cup _{x \in X} N(x)$. Also, we denote the neighbors of $X$ in $S$ by $N_S(X)=N(X) \cap S$. 
An {\it independent set} is a set of vertices in a graph such that no two of which are adjacent. The {\it independence number} $\alpha(G)$ is the size of a largest independent set in $G$. A {\it dominating set} of $G$  is a subset $D$ such that every vertex not in $D$ is adjacent to at least one vertex in $D$. The {\it domination number} $\gamma(G)$ is the size of a smallest dominating set in $G$. A subset $S \subseteq V(G)$ in which all components of $G\setminus S$ are cycles is called a {\it cycling set}. Moreover, if $S$ is an independent set, then we say that $S$ is an {\it independent cycling set}. We denote the number of path components of $G$ by $n(P,G)$.\\
A subset $M \subseteq E(G)$ is called a {\it matching}, if no two edges of $M$ are incident. A matching $M$ is called a {\it perfect matching}, if every vertex of $G$ is incident with some edge in $M$. Hall proved that a bipartite graph $G=(A,B)$ has a matching saturates $A$ if and only if for every $S \subseteq A$ we have $|N_B(S)| \geq |S|$, see \cite{kano} and \cite{Bondy}. 
\\
A graph $G$ has an $H$-{\it decomposition}, if all edges of $G$ can be decomposed into subgraphs isomorphic to $H$. If $G$ has an $H$-decomposition, then we say that $G$ is $H$-{\it decomposable}. A tree with exactly two non-pendant vertices is called a double-star. Let $k_1$ and $k_2$ be two positive integers. The double-star with degree sequence $(k_1+1, k_2+1, 1, \ldots, 1)$ is denoted by $S_{k_1, k_2}$. A vertex of degree $i$ is called an $i$-{\it vertex}. If $G$ is an $r$-regular graph with an $S_{1, r-1}$-decomposition and $S \subseteq V(G)$ is the set of all $r$-vertices of this decomposition, then we say that $G$ is $(S_{1, r-1},S)$-decomposable.  
\\
\begin{figure}[h]
\centering{\includegraphics[width=23mm]{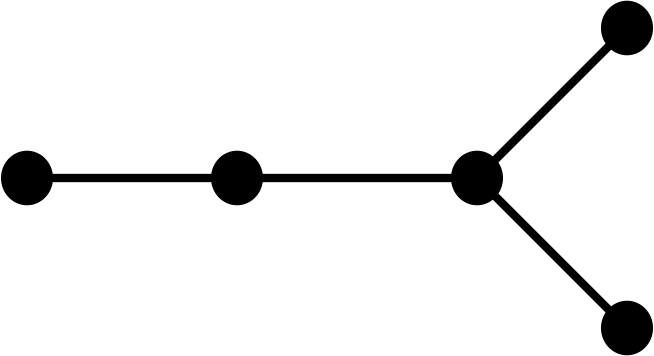}} \\
Figure1. $S_{1,2}$.
\end{figure}
\\
Tree decomposition of highly connected graphs is studied in \cite{Y}, \cite{Thomassen 1} and \cite{Thomassen 2}. In \cite{Y} it has been shown that every 191-edge-connected graph, whose size is divisible by 4 has an $S_{1, 2}$-decomposition.
\\
Let $G$ be a cubic graph. If $G$ is $S$-decomposable and $S$ is a double-star, then $S$ is isomorphic to $S_{1,1}$ or $S_{1,2}$ or $S_{2,2}$, because otherwise $S$ has a vertex of degree at least four. It was proved that a cubic graph has an $S_{1,1}$-decomposition if and only if it contains a perfect matching, see \cite{Kotzig}. In this paper, we study edge-decomposition of cubic graphs into copies of $S_{1,2}$. 
\\
This paper is organized as follows. In Section 2, we study $S_{1,2}$-decomposition of cubic graphs and provide some necessary and some sufficient conditions for the existence of an $S_{1,2}$-decomposition in cubic graphs. In Section 3, we obtain some results on $S_{1, r-1}$-decomposition of $r$-regular graphs.

\section{$S_{1, 2}$-Decomposition of Cubic Graphs}
In this section, we present some necessary and some sufficient conditions for the existence of $S_{1,2}$-decompositions in cubic graphs. Finally, we study $R$-decomposition in cubic graphs.
\\
Let $G$ be a cubic graph and $S \subseteq V(G)$. The question is that whether $G$ is $(S_{1,2},S)$-decomposable or not? For giving a response to this question, we need a new bipartite graph $H=(S,L)$, in which $S$ is the set of all 3-vertices of $S_{1,2}$-trees and for each edge $e \in E(G\setminus S)$, we put a vertex $u_e$ in $L$. Two vertices $s_i$ and $u_{e_j}$ are adjacent in $H$ if and only if there exists an edge $e \in E(G)$ such that one end of $e$ is $s_i$ and moreover $e$ and $e_j$ have a common end vertex. This means that $u_{e_j}$ and $s_i$ are adjacent in $H$ if and only if we can obtain an $S_{1,2}$ by adding $e_j$ to a claw containing $s_i$ as a central vertex. We have the following.

\begin{lem}\label{lem1}
Let $G$ be a cubic graph of order $n$. Then $G$ is $(S_{1,2},S)$-decomposable if and only if $|S|=\frac{3n}{8}$ and $H=(S,L)$ has a perfect matching.
\end{lem}

\begin{proof}
Clearly, if $G$ is $(S_{1,2},S)$-decomposable, then $|S|=\frac{3n}{8}$. Also, note that if $e_s=uv$ is an edge in some $S_{1,2}$ with $s$ as a 3-vertex and moreover $u$ and $v$ are 2-vertex and 1-vertex of this double-star, respectively. Then $e_s\in L$ and $\{(s,e_s): s\in S\}$ is a perfect matching in $H$.
\\ 
Conversely, if there exists a perfect matching $M=\{(s,e_{s}): s\in S\}$, then one can obtain an $S_{1,2}$-decomposition.
\end{proof}

In the following lemma we provide some necessary conditions for $S_{1,2}$-decomposition of cubic graphs.

\begin{lem}\label{lem2}
Let $G$ be a cubic graph of order $n$ which has a $S_{1,2}$-decomposition. Then the following hold:
\vspace{0.6em}
\\
(i)   $8\;|\;n$.
\vspace{0.4em}
\\
(ii) There exists an independent set $S \subset V(G)$ such that:

1- $|S| \geq \frac{3n}{8}$,

2- Each component of $G\setminus S$ is either a cycle or a tree,

3- No component of $G\setminus S$ has two $3$-vertices.
\vspace{0.4em}
\\
(iii) There exists an independent set $T\subseteq V(G\setminus S)$ with $|T|=\frac{n}{4}$.
\end{lem}

\begin{proof}
Let $G$ be $(S_{1,2},S)$-decomposable. Then from the preceding lemma, (i) and the first part of (ii) are clear.
\\
Suppose that $F$ is a given component of $G\setminus S$. If $F$ is neither a tree nor a cycle, then it has a cycle like $C:v_1,e_1,v_2,e_2, \ldots ,v_t,e_t,v_1$ and an edge $e=v_iw$, where $1 \leq i \leq t$ and $w\in V(F)$. Consider the set $A=E(C)$. Since $G$ is cubic, each vertex in cycle $C$ has at most one neighbor in $S$ and $v_i$ has no neighbor in $S$. Hence $|N_H(A)| \leq |A|-1$, which contradicts Hall's condition and so by Lemma \ref{lem1} $G$ has no $(S_{1,2},S)$-decomposition, a contradiction.
\\
If there exist two 3-vertices $u$ and $v$ in some component $F$, then there exists a $(u,v)$-path $P:u=v_1,e_1,\ldots,e_t,v_t=v$ in $F$. Now, let $A=E(P)$. Similar to the proof of the previous part, one can show that $|N_H(A)| \leq |A|-1$, which contradicts Hall's condition and so by Lemma \ref{lem1} $G$ has no $(S_{1,2},S)$-decomposition, a contradiction.
\\ 
For (iii), let $T$ be the set of vertices which are only used as a pendant vertex in $S_{1,2}$-trees in the given decomposition. It is easy to see that  $T$ is an independent set and $|T|=\frac{n}{4}$. Now, the proof is complete.
\end{proof}

These necessary conditions are not sufficient. Some examples are given as follows.
\\
\begin{figure}[h]
\centering{\includegraphics[width=90mm]{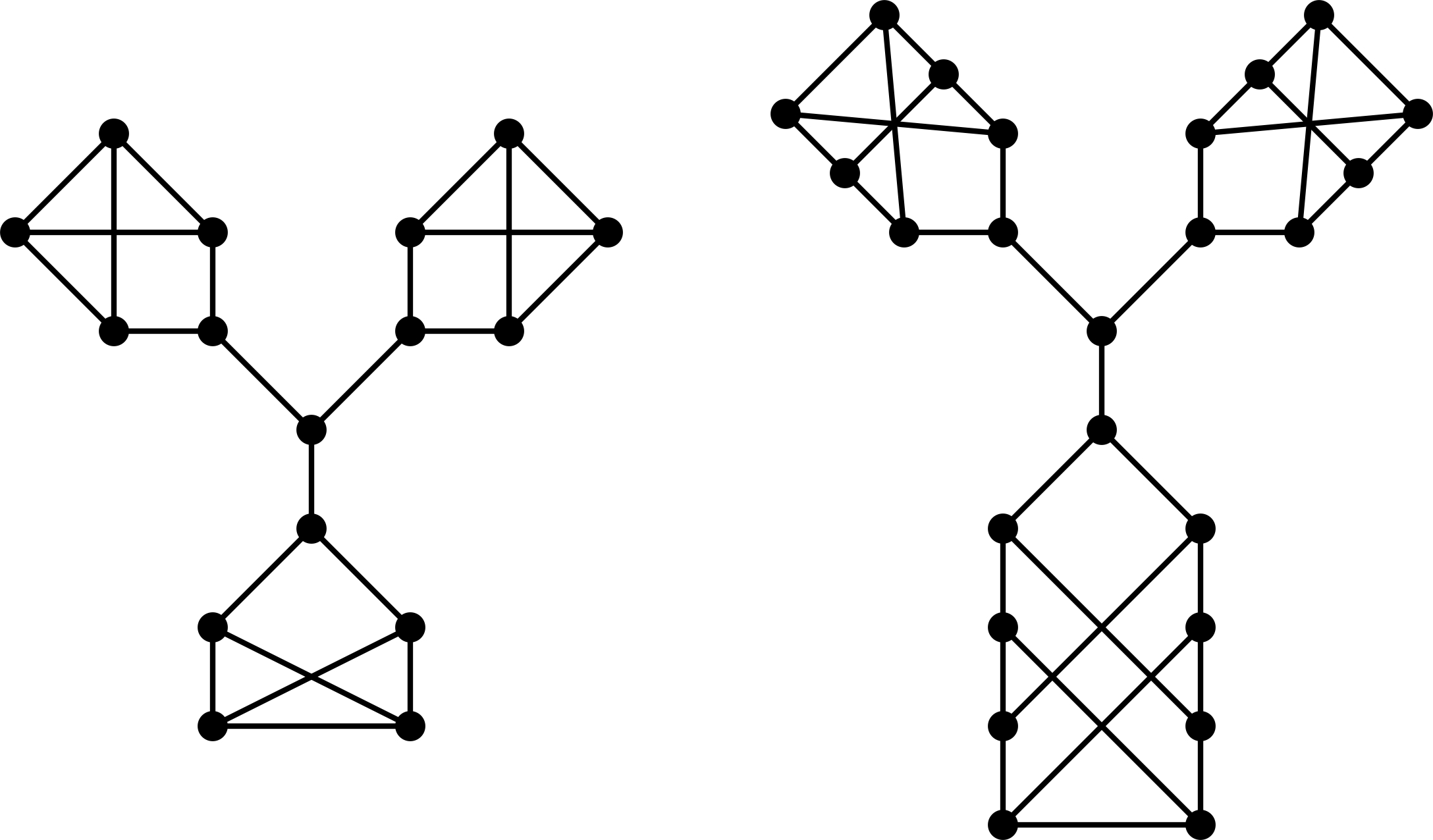}} \\
Figure 2.
\end{figure}
\\
Now, we provide some sufficient conditions for the existence of $S_{1,2}$-decomposition in cubic graphs and in the next section we will generalize them for the $S_{1,r-1}$-decomposition of $r$-regular graphs. 
\\
By Lemma \ref{lem2}, if $G$ is an $S_{1,2}$-decomposable cubic graph, then $\alpha(G) \geq \frac{3n}{8}$. We consider the case $\alpha(G)=\frac{3n}{8}$ and find two sufficient conditions for the existence of an $S_{1,2}$-decomposition in this case.

\begin{thm}\label{cycling}
Let $G$ be a cubic graph of order $n$ with $\alpha(G)=\frac{3n}{8}$. Suppose that there exists an independent cycling set $S \subseteq V(G)$ such that $|S|=\frac{3n}{8}$ and moreover, no vertex of $S$ is contained in a triangle. Then $G$ is $(S_{1,2}, S)$-decomposable.
\end{thm}

\begin{proof}
Suppose that $C_i$ ($1 \leq i \leq t$) are cycle components of $G\setminus S$. We claim that there are no two non-isolated vertices in $G\setminus S$ which have the same neighbor in $S$. Consider two non-isolated vertices $u$ and $v$ in $G\setminus S$. If $u$ and $v$ are adjacent, then since no vertex of $S$ is contained in a triangle, we are done. Now, suppose that $u$ and $v$ are not adjacent. If $N_S(u)=N_S(v)=\{s\}$, then $S'=(S\setminus \{s\}) \cup \{u,v\}$ is an independent set and $|S'|>\frac{3n}{8}$, a contradiction.
\\
Now, for each component $C:v_1,e_1,\ldots,v_t,e_t,v_1$ of $G\setminus S$ there exist distinct vertices $s_{i_1},\ldots,s_{i_t}$ in which $v_k$ is adjacent to $s_{i_k}$ in $S$. By adding $e_j$ to a claw containing $s_{i_j}$ as a central vertex we obtain an $S_{1,2}$-decomposition.  
\end{proof}

As a special case, we have the following result.

\begin{cor}
Let $G$ be a triangle-free cubic graph with $\alpha(G)=\frac{3n}{8}$ and there exists an independent cycling set $S \subseteq V(G)$ such that $|S|=\frac{3n}{8}$. Then $G$ is $(S_{1,2},S)$-decomposable.
\end{cor}

Another interesting result in the case of $\alpha(G)=\frac{3n}{8}$ is as follows.

\begin{thm}\label{main}
Let $G$ be a cubic graph of order $n$ with $\alpha(G)=\frac{3n}{8}$ and moreover there exists an independent set $S\subseteq V(G)$ such that $|S|=\frac{3n}{8}$ and no vertex of $S$ is contained in a triangle, $C_5$ or $C_7$. Then $G$ is $(S_{1,2},S)$-decomposable.
\end{thm}

\begin{proof}
We divide the proof into four claims.
\vspace{0.8em}
\\
\textbf{Claim 1.} Each  component of $G\setminus S$ is a path or a cycle.
\\
If there exists a vertex of degree 3 in $G\setminus S$, then by adding this vertex to $S$ we obtain an independent set $S'$ such that $|S'|=\frac{3n}{8}+1$,  a contradiction. 
\vspace{0.8em}
\\
\textbf{Claim 2.} Let $u,v\in V(G\setminus S)$ be two vertices of degree two in $G\setminus S$. Then $N_S(v)\neq N_S(u)$.
\\
Let $u$ and $v$ be two vertices in $G\setminus S$ such that $d_{G \setminus S}(u)=d_{G \setminus S}(v)=2$ and $N_S(u)=N_S(v)=\{s\}$. If $u$ and $v$ are adjacent, then $s$ is contained in a triangle,  a contradiction. Also, if $u$ and $v$ are not adjacent, then $S'=(S \setminus\{s\})\cup \{u,v\}$ is an independent set and  $|S'|=\frac{3n}{8}+1$, a contradiction.  So, the claim is proved.
\vspace{0.8em}
\\
Now, we check the Hall's condition for the edges of $G\setminus S$. Suppose that  $L=\{e_1,\ldots,e_l\}\subseteq E(G\setminus S)$. Let $P_1,\ldots,P_k$ be all path components of $G\setminus S$. Now, we consider two cases:
\vspace{0.8em}
\\
\textbf{Case 1.} No $P_i$ is contained in $\langle L \rangle$. 
\\
Note that for each edge  $e\in L$, one of its endpoints has degree 2 in $G\setminus S$. Because if both endpoints are of degree 1 in $G\setminus S$, then the induced subgraph on this edge is a path component of $G\setminus S$.  Now, we show that for each edge $e_i \in L$, one can find $v_{e_i} \in V(G \setminus S)$ such that $d_{G \setminus S}(v_i)=2$, $v_i$ is an endpoint of $e_i$ and if $i \neq j$, then $v_{e_i} \neq v_{e_j}$. 
\\
For $e_1$ define $v_{e_1}$ one of the its endpoints whose degree is 2. If $v_{e_1}$ is not one of the endpoints of $e_i$, $2 \leq i \leq l$, then define $v_{e_2}$ as one of its endpoints which has degree 2 in $G\setminus S$.  Otherwise, suppose that $v_{e_1}$  is one of the endpoints of $e_j=\{v_{e_1}, u\}$. If $d_{G\setminus S}(u)=1$, then $e_i$ and $e_j$ induce a $P_3$-component in $G\setminus S$, a contradiction. Hence,  $d_{G\setminus S}(u)=2$ and define $v_{e_j}=u$. By repeating this procedure for each edge $e\in L$ one can find $\{v_{e_1}, \ldots, v_{e_l}\}$. Now, Claim 2 implies that for each $e \in L$ there exists a distinct vertex in $S$ which is adjacent to $v_e$ and so in this case Hall's  condition holds. 
\vspace{0.8em}
\\
\textbf{Case 2.} There exist $i_1,\ldots,i_t$ such that $1\leq i_j\leq k$ and $P_{i_1},\ldots,P_{i_t}$ are all path components of $\langle L \rangle$. We have the following.
\vspace{0.8em}
\\
\textbf{Claim 3.} Let $v\in V(G\setminus S)$ such that $d_{G\setminus S}(v)=1$ and $N_{S}(v)=\{x,y\}$. Then both  $x$ and $y$ are not adjacent to the vertices of degree two in $G\setminus S$.
\\
Let $x$ and $y$ be adjacent to $v_x$ and $v_y$ in $G\setminus S$, respectively, and  $d_{G\setminus S}(v_x)=d_{G\setminus S}(v_y)=2$. Note that $v$ is not adjacent to $v_x$ and $v_y$, since otherwise there exists a triangle containing $v$, a contradiction. Now, if $v_x$ and $v_y$ are adjacent, then $C:v,x,v_x,v_y,y,v$ is a cycle of length 5, a contradiction. If $v_x$ and $v_y$ are not adjacent, then define $S'=(S\setminus \{x,y\})\cup \{v,v_x,v_y\}$. It can be easily seen that $S'$ is an independent set and  $|S'|=\frac{3n}{8}+1$,  a contradiction.\vspace{0.8em}
\\
Now, we can prove that in the second case, $L$ satisfies the Hall condition. It suffices to show that the edges of  $P_{i_1},\ldots,P_{i_t}$ satisfy Hall's condition. Because, similar to the proof of  the first case, one can see that other edges have distinct neighbors in $S$ and we are done. Now, Claim 2 implies that we can find $\sum_{j=1}^t(|E(P_{i_j})|-1)$ vertices in $S$ which are adjacent to the vertices of degree 2 in the path components. Let $T\subseteq S$ be the set of vertices in $S$ which are adjacent to the end vertices of $P_{i_1}\ldots,P_{i_t}$ and they are  adjacent to no vertex of degree 2 in $G\setminus S$. It suffices to show that $|T|\geq t$.
\\
By contrary, suppose that $|T|\leq t-1$. Then Claim 3 implies that each end vertex of paths has a neighbor in $T$. Let $A$ be the set of end vertices of paths that have one neighbor in $T$ and let $B$ be the set of end vertices which have two neighbors in $T$. We have the following.
\begin{center}
$|A|+|B|=2t \;\; , \;\;|A|+2|B|\leq 3t-3.$
\end{center}
Hence, we conclude that  $|A|\geq t+3$. Now, we prove the following claim.
\vspace{0.8em}
\\
\textbf{Claim 4.} If $u,v \in A$, then $N_T(u) \cap N_T(v) = \emptyset$.
\\
First, note that if $u$ and $v$ are adjacent, then we are done. So, we may assume that $u$ and $v$ are not adjacent. Let $N_T(u)=N_T(v)=\{w\}$. Suppose that $N_S(u)=\{w,x\}$ and $N_S(v)=\{w,y\}$. By the definition of $T$, we conclude that $x$ and $y$ are adjacent to some vertices of degree 2 in $G\setminus S$ say $v_x$ and $v_y$, respectively. Notice that if $x=y$ and $\{u, v, v_x\}$  is not independent set, then one can find a triangle contains a vertex of $S$, a contradiction. 
Thus, $\{u, v, v_x\}$ is an independent set. Now,  $S'=(S\setminus \{u,x\})\cup \{u,v,v_x\}$ is an independent set of size $\frac{3n}{8}+1$, a contradiction. Hence, $x\neq y$. We show that $\{u,v,v_x,v_y\}$ is an independent set. 
Since no vertex of $S$ is contained in a triangle, $u$ and $v_x$ are not adjacent (similarly, $v$ and $v_y$ are not adjacent). So, suppose that $v$ and $v_x$ are adjacent. Then $C:u,x,v_x,v,w,u$ is a cycle of length 5 which contains vertices of $S$, a contradiction. Also, note that $v_x$ and $v_y$ are not adjacent. Since, otherwise $C:u,x,v_x,v_y,y,v,w,u$ is a cycle of length 7, a contradiction. This implies that $\{u,v,v_x,v_y\}$ is an independent set. Now, let $S'=(S\setminus \{x,y,w\})\cup \{u,v,v_x,v_y\}$. Then $S'$ is an independent set and $|S'|>\frac{3n}{8}$, contradiction and this completes the proof of the claim.
\vspace{0.8em}
\\
Now, Claim 4 implies that for every $v\in A$ we have a distinct neighbor $t_v\in T$ and this implies that $|T|\geq t+3$, a contradiction. This completes the proof.
\end{proof}

Now, we have an immedaite corollary.

\begin{cor}
Let $G$ be a $\{C_3, C_5, C_7\}$-free cubic graph of order $n$ with $\alpha(G)=\frac{3n}{8}$. Then $G$ has an $S_{1,2}$-decomposition.
\end{cor}

Now, we obtain another sufficient condition for the existence of an $S_{1,2}$-decomposition in a cubic graph. If $G$ is a cubic graph, then $\gamma(G) \geq \frac{n}{4}$. In the following theorem, we provide a necessary and  sufficient condition on the existence of a $S_{1,2}$-decomposition for cubic bipartite graph $G$ under which $\gamma(G) = \frac{n}{4}$.

\begin{thm}
Let $G=(A,B)$ be a cubic bipartite graph of order $n$ such that $8|n$. Then $\gamma(G)=\frac{n}{4}$ if and only if there exists $S \subseteq A$ of size $\frac{3n}{8}$ such that $G$ is both $(S_{1,2}, S)$-decomposable and $(S_{1,2}, N(A\setminus S))$-decomposable.
\end{thm}

\begin{proof}
Let $D$ be a dominating set of $G$ of size $\frac{n}{4}$. Then vertices of $D$ has no common neighbors in $V(G)\setminus D$. Now, let $D_1=D \cap A$ and $D_2=D \cap B$ and $|D_1|=a,|D_2|=b$. Since $D$ is a dominating set of size $\frac{n}{4}$ we have:

\begin{center}
$a+b=\frac{n}{4}$ , $3a+b=\frac{n}{2}$.
\end{center}
Then $a=b=\frac{n}{8}$. Now, let $S=N(D_1)$. We show that $G$ has an $(S_{1,2}, S)$-decomposition. Clearly, $|S|= \frac{3n}{8}$ and $E(G\setminus S)$ is exactly the edges between $D_2$ and $N(D_2)$. Note that if $v \in N(D_2)$, then $d_{S}(v)=2$. Now, it is not hard to see that the graph $H=(S,L)$, defined in Lemma \ref{lem1}, is a 2-regular bipartite graph and hence it has a perfect matching. So, by Lemma \ref{lem1}, $G$ is $(S_{1,2}, S)$-decomposable.
Notice that if we consider $T=N(D_2)$, then similarly $G$ is $(S_{1,2}, T)$-decomposable. Since $T=N(A\setminus S)$, this completes the proof of the one side of the theorem.
\\
Conversely, suppose that there exists $S\subseteq A$ such that satisfies the conditions. Note that each vertex in $A\setminus S$ is a 3-vertex in $G\setminus S$. Now, Lemma \ref{lem2} implies that each of them is in a different component of $G\setminus S$ and so they have no common neighbors. By a similar method, one can show that the vertices of $B\setminus N(A\setminus S)$ have no common neighbors. Now, $D=(A\setminus S) \cup (B\setminus N(A\setminus S))$ is a dominating set of size $\frac{n}{4}$ and this completes the proof.  
\end{proof}

Now, we state the following corollary.

\begin{cor}
Let $G=(A,B)$ be a bipartite cubic graph of order $n$. If $\gamma(G)=\frac{n}{4}$, then $G$ is $S_{1,2}$-decomposable.
\end{cor}

As an example of this result we can obtain that $Q_3$ is $S_{1,2}$-decomposable, because it is bipartite and $\gamma(Q_3)= 2 = \frac{n}{4}$.

Now, we provide another sufficient condition for the existence of an $S_{1,2}$-decomposition in bipartite cubic graphs.

\begin{thm}
Let $G=(A,B)$ be a bipartite cubic graph of order $n$ and $S\subseteq A$ be of size $\frac{3n}{8}$. Then $G$ is $(S_{1,2},S)$-decomposable if and only if there exists a perfect matching between $S$ and $N(A\setminus S)$.
\end{thm} 

\begin{proof}
Necessity. First suppose that $G$ is $(S_{1,2},S)$-decomposable. Then the second part of Lemma \ref{lem2} indicates that no component of $G\setminus S$ has two 3-vertices. This implies that no two vertices of $A \setminus S$ have a common neighbor in $B$. So, $|N(A\setminus S)|=\frac{3n}{8}$. Now, note that if Hall's condition does not hold for $S$ and $N(A\setminus S)$, then Hall's condition does not hold in $H=(S, L)$, too. This is a contradiction and this completes the proof of the one side of theorem.
\\
Sufficiency. Suppose that there exists a perfect matching between $S$ and $N(A\setminus S)$. Then $|N(A\setminus S)|=\frac{3n}{8}$ which implies that no two vertices of $A \setminus S$ have a common neighbor in $B$. For each vertex $v\in N(A\setminus S)$, there exists a unique edge $e_v \in E(G\setminus S)$ in which $v$ is one of its end points. Let $M=\{(u_i,v_i)| \; i=1, 2, \ldots, \frac{3n}{8}\}$ be a matching between $S$ and $N(A\setminus S)$. Then by adding edge $e_{v_i}$ to a claw containing $u_i$ as a 3-vertex, one can obtain an $S_{1, 2}$-decomposition.
\end{proof}

\section{$S_{1, r-1}$-Decomposition of $r$-Regular Graphs}
In this section, we generalize the results of the previous section to $S_{1, r-1}$-decomposition of $r$-regular graphs.

Similar to the cubic graphs we find some necessary and some sufficient conditions for the existence of $S_{1,r-1}$-decomposition in $r$-regular graphs. The following holds.

\begin{preremark}
Let $G$ be an $r$-regular graph of order $n$ which is $S_{1, r-1}$-decomposable. Then $2(r+1)|rn$ and $\alpha(G) \geq \frac{rn}{2(r+1)}$.
\end{preremark}

The following theorem is a generalization of Theorem \ref{cycling}. But in this case we do not need the condition $\alpha(G)= \frac{rn}{2(r+1)}$.

\begin{thm}
Let $G$ be an $r$-regular graph ($r \geq 4$) of order $n$ and there exists an independent cycling set $S \subseteq V(G)$ such that $|S|= \frac{rn}{2(r+1)}$ and moreover no vertex of $S$ is contained in a triangle. Then $G$ is $(S_{1, r-1}, S)$-decomposable.
\end{thm}

\begin{proof}
We check Hall's condition for $H=(S,L)$ defined in Lemma \ref{lem1}. Let $e=uv \in E (G \setminus S)$. Since $G$ is $r$-regular and $S$ is a cycling set, each end of $e$ has exactly $r-2$ neighbors in $S$. No vertex of $S$ is contained in a triangle and this yields that $u$ and $v$ have no common neighbor in $S$. So, we conclude that $|N_S(\{u,v\})| = 2r-4$. Now, let $M=\{e_1,\ldots,e_t\} \subseteq E(G \setminus S)$ and $V_1$ be the set of all end points of the edges of $M$. We have:
\begin{center}
$|N_S(V_1)| \geq \frac{(2r-4)t}{r}.$
\end{center}
Now, since $r \geq 4$, $2r-4 \geq r$ and this implies that in $H=(S,L)$ we have $|N_S(M)| \geq |M|$, for every $M \subseteq L$ . So, $H=(S,L)$ satisfies the Hall condition and this yields that $G$ is $S_{1, r-1}$-decomposable. 
\end{proof}

We have the following result in regular bipartite graphs which makes a connection between the domination number and the existence of $S_{1, r-1}$-decomposition. The proof is the same as the case of cubic bipartite graphs and we omit it.

\begin{thm}
Let $G$ be a bipartite $r$-regular graph of order $2n$ such that $r+1|n$ and $\gamma(G)=\frac{2n}{r+1}$. Then $G$ is $S_{1, r-1}$-decomposable.
\end{thm}

\section{Questions}

In Theorem \ref{main}, we show that every $\{C_3, C_5, C_7\}$-free graph of order $n$ with $\alpha= \frac{3n}{8}$ is $S_{1,2}$-decomposable. In \cite{independent}, it was shown that if $G$ is a planar triangle-free graph with maximum degree at most $3$, then $\alpha(G) \geq \frac{3n}{8}$. There are several examples with $\alpha = \frac{3n}{8}$, containing $C_5$ or $C_7$ and have an $S_{1,2}$-decomposition in \cite{independent}. 
\begin{figure}[h]
\centering{\includegraphics[width=50mm]{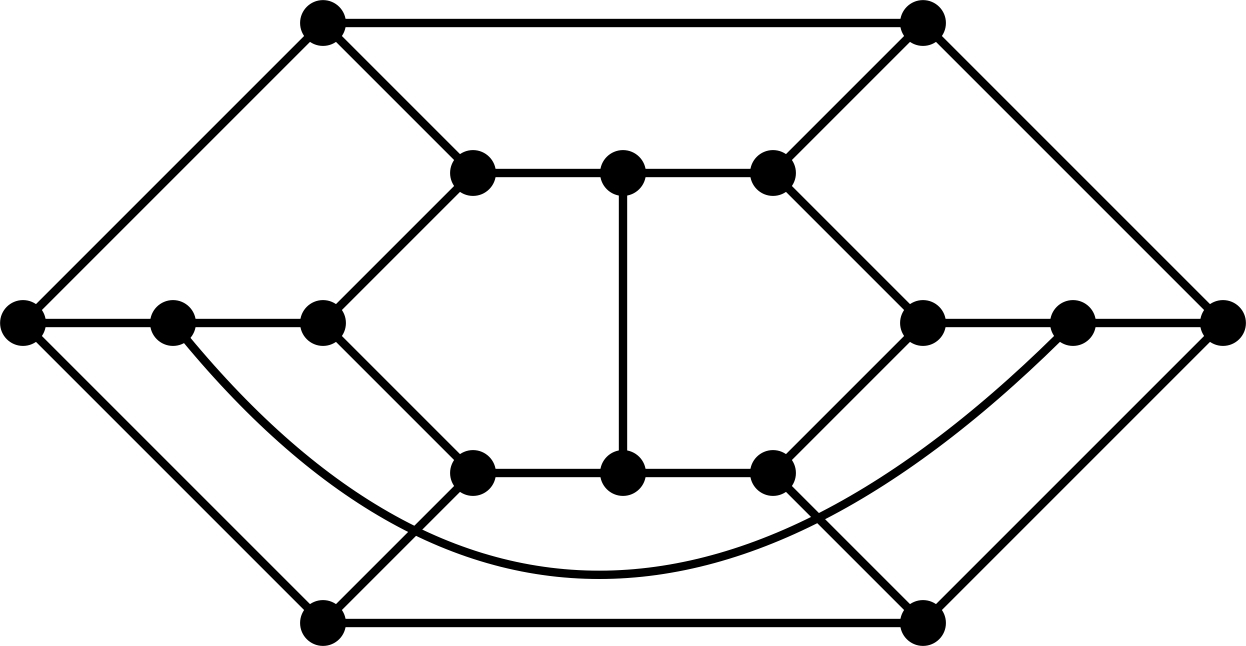}} \\

Figure 3.
\end{figure}
\\
Figure 3 shows that the conditions of being $\{C_5, C_7\}$-free in Theorem \ref{main} are not necessary.
\\
Also, there exists a $\{C_5, C_7\}$-free graph with $\alpha= \frac{3n}{8}$, containing a triangle and has no $S_{1,2}$-decomposition, see Figure 4. 
\begin{figure}[h]
\centering{\includegraphics[width=65mm]{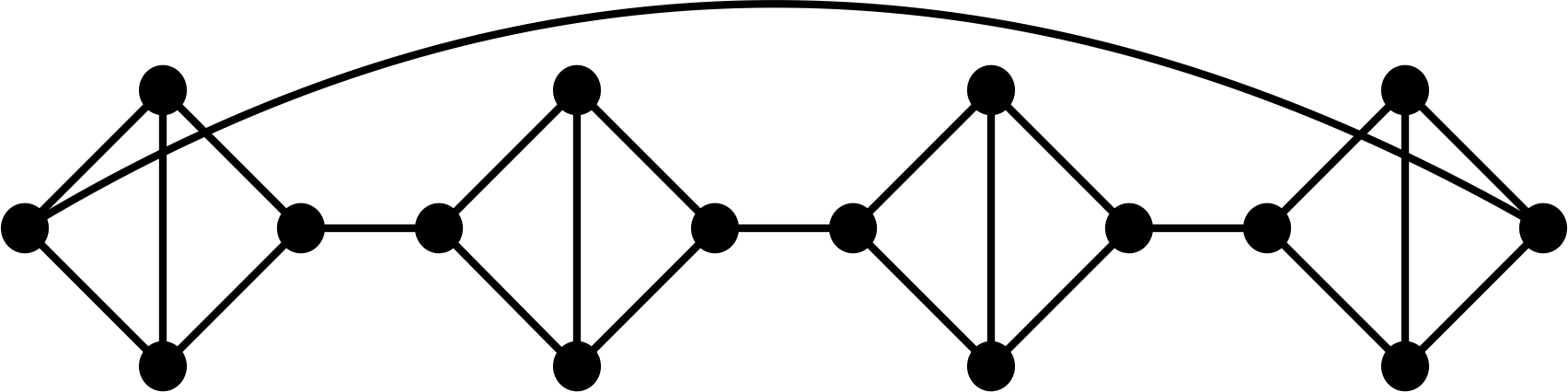}} \\

Figure 4.
\end{figure}
\\
This shows that the condition of being triangle-free in Theorem \ref{main} is necessary. Now, the following question is natural. 
\\
\textbf{Question 1.} Let $G$ be a triangle-free cubic graph of order $n$ with $\alpha(G) =\frac{3n}{8}$. Is it true that $G$ is $S_{1,2}$-decomposable?
\vspace{0.8em}
\\
Figure 2 and Figure 4 show that being connected and 2-connected are not sufficient. Also, Figure 2 shows that there exists a triangle-free connected cubic graph with no $S_{1,2}$-decomposition.
\
\\
\textbf{Question 2.} Does there exist a triangle-free 2-connected cubic graph of order divisible by 8 which has no $S_{1,2}$-decomposition?
\\
\textbf{Question 3.} Does there exist a (triangle-free) 3-connected cubic graph of order divisible by 8 which has no $S_{1,2}$-decomposition?
\vspace{0.6em}
\\
Another open question is as follows.
\vspace{0.6em}
\\
\textbf{Question 4.} Is it true that every bipartite cubic graph of order divisible by 8 is $S_{1,2}$-decomposable?

\end{document}